\documentclass{amsart}

\usepackage{epsfig,color}

\usepackage[latin1]{inputenc}

\usepackage[T1]{fontenc}

\usepackage{color}

\usepackage{pst-plot,pstricks} 
\usepackage[brazil,english]{babel}
\usepackage{ amsmath, amssymb}
\usepackage{bbm}

\usepackage{constants} 
\usepackage[colorlinks]{hyperref}

\hypersetup{ colorlinks=true, pdftitle={Central limit theorem for generalized Weierstrass functions}, pdfsubject={ dynamical systems, circle maps, expanding maps, ergodic theory, central limit theorem, weierstrass, nowhere differentiable, Takagi}, pdfauthor={Amanda de Lima and Daniel Smania},pdfkeywords={ dynamical systems, circle maps, expanding maps, ergodic theory, central limit theorem, bloch function, weierstrass, nowhere differentiable, Takagi}}

\newtheorem{theorem}{Theorem}[section]
\newtheorem{lemma}[theorem]{Lemma}
\newtheorem{corollary}[theorem]{Corollary}

\newtheorem{proposition}[theorem]{Proposition}

\newtheorem{remark}[theorem]{Remark}
\newtheorem{definition}[theorem]{Definition}

\newcommand{\Z}{\mathbb{Z}}
\newcommand{\R}{\mathbb{R}}
\newcommand{\N}{\mathbb{N}}

\newcommand{\s}{\mathbb{S}}

\begin{document}
\author{Amanda de Lima and Daniel Smania}
\address{
Departamento de Matem\'atica,
ICMC-USP, Caixa Postal 668,
CEP 13560-970, 
S\~ao Carlos-SP, Brazil.}
\email{smania@icmc.usp.br\\
amandal@icmc.usp.br}
\urladdr{http://conteudo.icmc.usp.br/pessoas/smania/}

\date{\today}
\title[CLT for generalized Weierstrass functions]{Central Limit Theorem for generalized Weierstrass functions}

\thanks{We thank the referees for the very helpful  comments. A.L. was partially supported by FAPESP 2010/17419-6 and D.S. was partially supported by CNPq 305537/2012-1, 307617/2016-5 and FAPESP 2017/06463-3}
\maketitle
\begin{abstract}  
Let $f: \s^1 \to \s^1$ be a $C^{2+\epsilon}$ expanding map of the circle and let $v: \s^1 \to \R$ be a $C^{1+\epsilon}$ function. Consider the twisted cohomological equation 
$$
v(x) = \alpha(f(x)) - Df(x)\alpha(x),
$$
which has a unique bounded solution $\alpha$. We show that $\alpha$ is either $C^{1+\epsilon}$ or nowhere differentiable. Moreover  $\alpha$ is nowhere differentiable if and only if $\sigma=\sigma(\phi)\neq 0$,  where 
$$
\phi(x)=- \left(\frac{v'(x)+\alpha(x)D^2f(x)}{Df(x)}\right)
$$
and
$$
\sigma^2(\phi) = \lim_{n\to \infty}\int \left(\frac{\sum_{j=0}^n \phi \circ f^j}{\sqrt{n}} \right)^2 \ d\mu.
$$
Here $\mu$ is the unique absolutely continuous invariant probability of $f$.  We show that if $\alpha$ is nowhere differentiable then 
$$
\lim_{h\to 0} \mu \left\{x:\frac{\alpha(x+h)-\alpha(x)}{\sigma \ell h\sqrt{-\log |h|}} \le y\right\} = \frac{1}{\sqrt{2\pi}} \int_{-\infty}^{y} e^{-\frac{t^2}{2}} dt.
$$
for some constant $\ell > 0$. In particular $\alpha$ is not a Lipschitz continuous  function on any subset with positive Lebesgue measure.
\end{abstract}

\section{Introduction and statement of the results}
\label{intro}
In the beginning of the nineteenth century, it was a common belief that a continuous function is differentiable at most of its domain. However, in 1872, Karl Weierstrass gave a stunning example of a function which is everywhere continuous but nowhere differentiable
\begin{equation}\label{wei}
W(x) = \sum_{k=0}^{\infty}a^k\cos(b^k\pi x).
\end{equation}
Here $a$ is a real number with $0<a<1$, $b$ is an odd integer and $ab> 1+ 3\pi/2$. In 1916, G. H. Hardy \cite{hardy} proved that the function $W$ defined above is continuous and nowhere differentiable if $0<a<1$, $ab\ge 1$. The constant $b$ does not need to be an integer.

 There are many contributions  on this subject after the introduction of the  Weierstrass function.  In 1903, T. Takagi \cite{takagi} presented  another example of a nowhere differentiable continuous function,  now called Takagi function, given by
$$
        T(x)=\sum_{k=0}^{\infty} \frac{1}{2^k}\,\inf_{m\in \Z}|2^kx - m|.
$$
Observe that  Weierstrass (in the case $a=b \in \mathbb{N}$) and Takagi functions satisfy the cohomological equations
$$
v(x) =  W(ax)-\frac{1}{a}W(x) , \mbox{\;\; where \;\;}  v(x) =- \frac{\cos(\pi x)}{a}
$$
and
$$
v(x) = T(2x)-2 T(x) , \mbox{\;\; where \;\;}  v(x) = -2\inf_{m\in \Z}|x-m|,
$$
respectively. It turns out these cohomological equations  are particular cases of the {\it twisted cohomological equation}, that is an essential tool in the study of the smooth perturbations of one-dimensional dynamical systems and the linear response problem (see Baladi and Smania  \cite{bs1} \cite{bs2} and Lyubich \cite{mhc}, for instance).  We can ask if similar results to those by Weiertrass and Takagi holds for more general dynamical systems $f$ and observables $v$. There are some results in this direction by Heurteaux \cite{yanick}, but only when $f$ is a {\it linear}   dynamical system.

Let us denote by $\s^1$ the unit circle, $\s^1 = \{(x,y) \in \R^2:x^2+y^2=1\}$.
Let $f: \s^1 \to \s^1$ be an expanding map, that is,  there is $\lambda >1$ such that
\begin{equation}\label{const}
|Df(x)|> \lambda,
\end{equation}
for every $x \in \s^1$ and let us consider the {\it twisted cohomological equation}
\begin{equation}\label{tce}
v(x)= \alpha(f(x)) - Df(x)\alpha(x),
\end{equation}
where $v: \s^1 \to \R$ and $\alpha: \s^1 \to \R$  are  bounded functions. 

In Baladi and Smania \cite{bs1} it was proved that the unique bounded function $\alpha$ satisfying  (\ref{tce}) is given by the formula
\begin{equation}\label{def_alpha}
\alpha(y) = -\sum_{n=1}^{\infty}\frac{v(f^{n-1}(y))}{Df^n(y)}.
\end{equation} 

Due to the similarity of this expression with the Weierstrass function we will call $\alpha$ a {\it generalized Weierstrass function}. Heurteaux \cite{yanick} considered the case when $f$ is a linear function and $v$ is an almost periodic function defining
\begin{equation}\label{type}
F(x) = \sum_{n=0}^{\infty}b^{-n}v(b^nx),
\end{equation}
where $1<b<\infty$. This function was called a Weierstrass-type function.

Along this work we will always assume 
\begin{align*}
{\mbox{\bf (A)\;\;\; }}&f: \s^1 \to \s^1 \text{ is a $C^{2+\epsilon}$ expanding map},\\
&v: \s^1 \to \R \text{ is a $C^{1+\epsilon}$ function, and }\\
&\alpha: \s^1 \to \R \text{  is the  {\it generalized Weierstrass function} as  defined in (\ref{def_alpha})}.
\end{align*}

 Our first results  about generalized Weierstrass functions are 

\begin{proposition}[Zygmund regularity] \label{prop_zygmund}
The function  $\alpha$ is in Zygmund class $\Lambda_1$, that is, there is $C>0$ such that
$$
|\alpha(x+h) + \alpha(x-h) -2\alpha(x)|\le C|h|,
$$
for all $x \in \s^1$ and $|h| \le 1$.
\end{proposition}
and
\begin{theorem}[Regularity Dichotomy] \label{teo_alpha_dif}
One of the following statements holds:
\begin{itemize}
\item[(i)] $\alpha$ is of class $C^{1+\epsilon}$;
\item[(ii)] $\alpha$ is nowhere differentiable.
\end{itemize}
\end{theorem}
Heurteaux  \cite{yanick}  proved a similar result for Weierstrass-type functions, when $f$ is a linear dynamical system. In the last section we give a very short proof of this result for non-linear dynamical systems. The above results  are not our main results, however  most of this work is dedicated to understand when the cases  in Theorem \ref{teo_alpha_dif} occurs, providing  an easy way to verify on which  case a given $v$ and $\alpha$ fit and also to give a more precise understanding of the regularity of $\alpha$ in the second case. Let
$$
\phi(x)=- \left(\frac{v'(x)+\alpha(x)D^2f(x)}{Df(x)}\right).
$$
Note that by Proposition \ref{prop_zygmund}, the function $\phi$ is $\epsilon$-H\"older. Since $f$ is an $C^2$ expanding map of the circle, $f$ admits a unique  invariant probability that is  absolutely continuous  with respect to the Lebesgue measure $\mu=\rho m$ (see for instance \cite{viana}), where its density $\rho$ is a H\"older function.  We show in Lemma \ref{lema_pressao} that
\begin{equation}\label{zeroint33}
\int \phi \ d\mu = 0.
\end{equation}
 Let us consider
$$
\sigma^2(\phi) = \lim_{n\to \infty}\int \left(\frac{\sum_{j=0}^n \phi \circ f^j}{\sqrt{n}} \right)^2 \ d\mu.
$$

We have the following

\begin{theorem}\label{teo_se_e_so_se}
We have  
$\sigma(\phi) =0$ if and only if $\alpha$ is of class $C^{1+\epsilon}.$
\end{theorem}

\begin{corollary}\label{cor_alpha}
The function  $\alpha$ is nowhere differentiable if and only if there is a periodic point $x$ of $f$ such that
$$
\sum_{i=0}^{p-1} \phi\circ f^i(x) \neq 0.
$$
Here  $p$ is the period of $x$.
\end{corollary}

\begin{remark}{\normalfont If $f$ is  the linear map $f(x)=bx$  then $\phi= -v'/b$. In this case Heurteaux \cite{yanick} proved that either 
\begin{itemize}
\item[i.] $\phi(0)\neq 0$, or
\item[ii.] $\phi(0)= 0$, $v$ is not constant and $0$ is a global  extremum of $v$,
\end{itemize}
are  sufficient conditions for  $\alpha$ to be nowhere differentiable.  This criterion was generalized to Weierstrass functions in higher dimensions by Donaire, Llorente and Nicolau \cite{dln}. The sufficiency of (i)  follows from Corollary \ref{cor_alpha} in the setting of $1$-periodic $C^{1+\epsilon}$ functions $v$.}
\end{remark}

Denote 
$$
L=\int \log|Df| \ d\mu>0, \ \ell = \frac{1}{\sqrt{L}}.
$$

One can ask about the regularity of $\alpha$ when it is nowhere differentiable. We show a Central Limit Theorem for the modulus of continuity of the function $\alpha$.
\begin{theorem}[Central Limit Theorem for the modulus of continuity] \label{teo_modulo_alpha}
If $\sigma(\phi)\neq 0$ (that is, $\alpha$ is nowhere differentiable) we have 
\begin{equation}\label{ccl}
\lim_{h\to 0} \mu \left\{x \in \mathbb{S}^1 :\frac{\alpha(x+h)-\alpha(x)}{\sigma(\phi) \ell h\sqrt{-\log |h|}} \le y\right\} = \frac{1}{\sqrt{2\pi}} \int_{-\infty}^{y} e^{-\frac{t^2}{2}} dt.
\end{equation}
\end{theorem}

\begin{corollary}\label{cor_lipschitz}
If $\alpha$ is not $C^{1+\epsilon}$ then $\alpha$  is nowhere differentiable and it is not a lipschitzian function on any measurable subset $A\subset \mathbb{S}^1$ with $\mu(A)>0$. 
\end{corollary}

We also have

\begin{theorem}[Law of iterated logarithm for the modulus of continuity] \label{teo_lil}
If  $\sigma(\phi)\neq 0$ (that is, $\alpha$ is nowhere differentiable) we have 
\begin{equation}\label{lii}
\limsup_{h\rightarrow 0} \frac{\alpha(x+h)-\alpha(x)}{ h\sqrt{-2\log |h| \log \log (-\log |h|) }} = \sigma(\phi)\ell .
\end{equation}
and
$$ 
\liminf_{h\rightarrow 0} \frac{\alpha(x+h)-\alpha(x)}{ h\sqrt{-2\log |h| \log \log (-\log |h|) }} = -\sigma(\phi)\ell .
$$
for $m$-almost every point $x$.
\end{theorem}

Gamkrelidze (see \cite{gamk1} and \cite{gamk2}) proved that the  Weierstrass and Takagi functions satisfy the Central Limit Theorem and the law of iterated logarithm  for the modulus of continuity. Theorems \ref{teo_modulo_alpha}  and  \ref{teo_lil} are   far more general results.

Dynamical systems with a fair amount of hyperbolicity  often have  remarkable probabilistic properties. Indeed in many cases a sufficiently smooth observable satisfies the Central Limit Theorem (CLT) and  the Law of Iterated Logarithm.  We also have Poisson Limit Theorems in this setting (see Coelho and Collet \cite{cc} and Denker, Gordin, and Sharova \cite{dgs}). For instance, the CLT holds for   $C^2$ expanding maps on circle with sufficiently regular observables, and indeed  the same holds for  piecewise expanding maps on the interval (see Keller \cite{kellerbv} and Hofbauer and Keller \cite{hofbauer}).  Theorems \ref{teo_modulo_alpha} and \ref{teo_lil} fits well with these probabilistic results.

On the other hand  note that if it is true that a fairly wide class of dynamical systems have observables that satisfies CLT (see Burton and Denker\cite{bd}), if the dynamical system does not have a hyperbolic behaviour (as  irrational rotations on the circle)  and/or we consider  low regularity observables, the Central Limit Theorem often fails in  quite a  striking way for "typical" continuous observables (see Liardet and Voln\'y \cite{lv}). 

Indeed for rotations on the circle  $f$ with diophantine rotation number and $C^k$-functions $v$ with zero average and  $k$ large enough, the solution $\alpha$  of the cohomological equation (\ref{tce}) exists and it is {\it always} smooth (see for instance Herman \cite{he1}\cite{he2} and the references therein).

\begin{remark}{\normalfont    Corollary \ref{cor_alpha}  implies  that for  a fixed $C^{2+\epsilon}$ expanding map and a generic ({\it in the topological sense}) $C^{1+\epsilon}$ function $v$  we have that (\ref{ccl}) and (\ref{lii}) hold. This  resembles the   L\'evy's modulus of continuity theorem \cite{levy} that gives a CLT and law of iterated logarithm for the  modulus of the continuity  of a typical (in the measure-theoretical sense)  path in Wiener process.}
\end{remark}

\begin{remark} {\normalfont One should compare Theorem \ref{teo_lil} with Theorem 1 in  Anderson and Pitt  \cite[Theorem 4.1]{ap}  that claims  that 
$$g(x)=\limsup_{h\rightarrow 0} \frac{|\alpha(x+h)-\alpha(x)|}{ |h|\sqrt{-\log |h| \log \log (-\log |h|) }} $$
is an essentially bounded function   for every function $\alpha$ in the Zygmund class.}\end{remark}

\begin{remark} {\normalfont A Bloch function  on the unit disk $\mathbb{D}$ is a holomorphic function $\beta\colon \mathbb{D}\rightarrow \mathbb{C}$ such that
$$\sup_{z \in \mathbb{D}}  |\beta'(z)|(1-|z|^2) < \infty.$$
Primitives of   Bloch functions extend continuously to $\mathbb{S}^1$  as functions of  Zygmund's class (see  Anderson, Clunie and Pommerenke  \cite{anderson} and Anderson and Pitt \cite{ap2} ). There is a long line of studies of  probabilistic-like properties of   those functions, and  Makarov's law of iterated logarithm \cite{makarov} is one of the most famous results on this topic. Indeed  many  authors  discovered various quantities associated to a Bloch function that resembles the role of variance in the corresponding probabilistic results.  It turns out that for  certain  "dynamically defined" Bloch functions (see Ivrii \cite{oleg}) all those variances  coincide and the variance has a dynamical interpretation (Przytycki,  Urba{\'n}ski and  Zdunik \cite{puz1}\cite{puz2}. See also   McMullen \cite{mac}and   Ivrii \cite{oleg}). That is the case of Bloch functions with  a  primitive whose extension to $\mathbb{S}^1$  is a function $\alpha$ as in this work, when the variance is   $\sigma(\phi)\ell$. So Theorem \ref{teo_lil} seems to offer yet another way to define variance of a Bloch function on the unit disk, but this time just in terms of its primitive. If $\alpha$  is the primitive of a Bloch function $\beta$ in the unit circle, we define its boundary  variance $\sigma_{b}(\beta)$ as 

$$\sigma_{b}(\beta) = || \limsup_{h\rightarrow 0} \frac{|\alpha(x+h)-\alpha(x)|}{ |h|\sqrt{-\log |h| \log \log (-\log |h|) }}||_{L^\infty(\mathbb{S}^1)}.$$
}
 \end{remark}


\section{Relating the Newton quotient of $\alpha$ with the Birkhoff sum of $\phi$}

In order to prove Theorem \ref{teo_modulo_alpha}, in this section we  will relate  the study of the  Newton quotients of $\alpha$
$$\frac{\alpha(x+h)-\alpha(x)}{h}$$
to the study of the Birkhoff sums
$$\sum_{i=0}^{N}\phi \circ f^{i}(x).$$

\begin{remark}\label{sim} {\normalfont  Suppose that the topological degree of $f$ is $d$. To simplify the notation, we will replace $f$, $v$  and $\alpha$ by its lifts with respect to the covering $\pi(t)=(\cos (2\pi t),\sin (2\pi t))$. That is, we will see $f$ as an expanding function $f\colon \mathbb{R} \mapsto \mathbb{R}$ satisfying $f(x+1)=f(x)+d$ and $v$ and $\alpha$ as  $1$-periodic functions $v\colon \mathbb{R} \rightarrow \mathbb{R}$ and $\alpha\colon \mathbb{R} \rightarrow \mathbb{R}$. }\end{remark}

\begin{definition}\label{deff}
Given  $h$ such that $0 < |h| < 1$ and $x \in \s^1$, let $N(x,h)$ be the unique integer such that
\begin{equation}\label{defNxh}
\frac{1}{|Df^{N(x,h)+1}(x)|}\le |h| <  \frac{1}{|Df^{N(x,h)}(x)|}.
\end{equation}
\end{definition}

The main result of this section is 
\begin{proposition}\label{prop_alpha}
Let $N(x,h)$ be as defined in (\ref{defNxh}). Then
$$
\alpha(x+h)-\alpha(x) = h\sum_{i=0}^{N(x,h)-1} \phi(f^{i}(x)) + O(h),
$$
where
$$
\phi(x)=-\left(\frac{v'(x)}{Df(x)} + \frac{\alpha(x)D^2f(x)}{Df(x)}\right).
$$
\end{proposition}

Before proving Proposition \ref{prop_alpha} we will need some lemmas. The following is a quite familiar bounded distortion estimate.

\begin{lemma}[Bounded Distortion]\label{lema_distorcao}
Denote also by  $f$ be the lift to $\mathbb{R}$ of the function $f$.  Then there is  $C>0$ such that for all $n \in \N$
$$
\left|\frac{Df^n(x)}{Df^n(y)}\right| \le e^{C|f^n(x)-f^n(y)|}.
$$
\end{lemma}

The following  lemma is an easy consequence of Lemma \ref{lema_distorcao}.

\begin{lemma}\label{size2}
There exists $C >1$ such that for every $x \in \mathbb{R}$ and $h$ satisfying $0< |h|<  1$  we have 
\begin{equation} \label{size}
\frac{1}{C} \leq |f^{N(x,h)}[x,x+h]|\le C.
\end{equation} 
\end{lemma}

\begin{lemma}\label{bb666} There exists $C > 0$ such that for every $h$ satisfying  $0< |h| < 1$ and  for every $j\leq n\leq N(x,h)$ 
$$|f^n(a) -f^n(b)|\leq \frac{C}{\lambda^{N(x,h)-n}},$$
and
$$\Big|  \frac{Df^{n-j}(f^ja)}{Df^{n-j}(f^jb)} -1 \Big|\leq \frac{C}{\lambda^{N(x,h)-n}},$$
for every $a,b \in [x,x+h]$ and $\lambda$ is as in (\ref{const}).
\end{lemma}
\begin{proof}It follows from Lemma \ref{lema_distorcao} and Lemma \ref{size2}.
\end{proof}

\begin{proof}[Proof of Proposition \ref{prop_alpha}]
By Lemma \ref{bb666}, for every $y \in [x,x+h]$  we have 
\begin{align*}
\alpha(y) &= -\sum_{n=1}^{N(x,h)}\frac{v(f^{n-1}(y))}{Df^n(y)} -\frac{1}{Df^{N(x,h)}(y)} \sum_{n> N(x,h)}^{\infty}\frac{v(f^{n-1}(y))}{Df^{n-N(x,h)}(f^{N(x,h)}(y))}   \\
&= -\sum_{n=1}^{N(x,h)}\frac{v(f^{n-1}(y))}{Df^n(y)} + O(h),
\end{align*} 
since by Definition \ref{deff}, Lemmas \ref{lema_distorcao} and \ref{size2} 
\begin{align*} &  \big| \frac{1}{Df^{N(x,h)}(y)} \sum_{n> N(x,h)}^{\infty}\frac{v(f^{n-1}(y))}{Df^{n-N(x,h)}(f^{N(x,h)}(y))} \big| \\
&\leq \big| \frac{1}{Df^{N(x,h)}(x)} \big| \big| \frac{Df^{N(x,h)}(x)}{Df^{N(x,h)}(y)}  \sum_{n> N(x,h)}^{\infty}\frac{v(f^{n-1}(y))}{Df^{n-N(x,h)}(f^{N(x,h)}(y))} \big| \\
&\leq  \frac{|h|}{\lambda}  e^{C^2} \sum_{i=1}^\infty  \frac{|v|_\infty}{\lambda^i}.
\end{align*} 
So by the Mean Value Theorem there is $\theta \in [x,x+h]$ such that
\begin{align} \label{pw1}
&\alpha(x+h)-\alpha(x) \nonumber \\
&= -h  \sum_{n=1}^{N(x,h)}\frac{Dv(f^{n-1}(\theta)) Df^{n-1}(\theta) Df^n(\theta) -    D^2f^n(\theta)v(f^{n-1}(\theta))}{[Df^n(\theta)]^2}  + O(h) \nonumber  \\
&= -h  \sum_{n=1}^{N(x,h)} \big[ \frac{Dv(f^{n-1}(\theta))}{Df(f^{n-1}(\theta))} -    \frac{D^2f^n(\theta)v(f^{n-1}(\theta))}{[Df^n(\theta)]^2}\big]   + O(h)
\end{align} 
Note that 
\begin{align} \label{pw2}
  \sum_{n=1}^{N(x,h)} \frac{D^2f^n(\theta)v(f^{n-1}(\theta))}{[Df^n(\theta)]^2}&=  \sum_{n=1}^{N(x,h)}  \sum_{j=0}^{n-1} \frac{D^2 f(f^j(\theta))v(f^{n-1}(\theta))}{Df(f^j(\theta)) Df^{n-j}(f^{j}(\theta))}  \nonumber\\
  &=  \sum_{j=0}^{N(x,h)-1} \frac{D^2 f(f^j(\theta))}{Df(f^j(\theta))}    \sum_{n=j+1}^{N(x,h)} \frac{v(f^{n-j-1}(f^j(\theta)))}{Df^{n-j}(f^{j}(\theta))} \nonumber\\
  &=  \sum_{j=0}^{N(x,h)-1} \frac{D^2 f(f^j(\theta))}{Df(f^j(\theta))}    \sum_{k=1}^{N(x,h)-j} \frac{v(f^{k-1}(f^j(\theta)))}{Df^{k}(f^{j}(\theta))}\nonumber \\
   &=  \sum_{j=0}^{N(x,h)-1} \frac{D^2 f(f^j(\theta))}{Df(f^j(\theta))}  [- \alpha(f^j(\theta))+\frac{\alpha(f^{N(x,h)}(\theta))}{Df^{N(x,h)-j}(f^j(\theta))}].
\end{align} 
and
\begin{align} \label{pw3}
\sum_{j=0}^{N(x,h)-1} \big| \frac{D^2 f(f^j(\theta))}{Df(f^j(\theta))}  \frac{\alpha(f^{N(x,h)}(\theta))}{Df^{N(x,h)-j}(f^j(\theta))}\big| &\leq \sum_{j=0}^{N(x,h)-1} \frac{C}{\lambda^{N(x,h)-j}} \leq \frac{C\lambda}{\lambda-1}.
\end{align} 
By  (\ref{pw1}), (\ref{pw2}) and (\ref{pw3}) we conclude that 
\begin{align} \label{pw11}
\alpha(x+h)-\alpha(x) \nonumber =  h  \sum_{j=0}^{N(x,h)-1}\phi(f^j(\theta))   + O(h) \nonumber  
\end{align} 
By Baladi and Smania \cite{bs3} we know that $\alpha$, and consequently $\phi$  is a $\gamma$-H\"older function for every $\gamma \in (0,1)$, so Lemma \ref{bb666} easily  implies that 
$$\sum_{j=0}^{N(x,h)-1} ( \phi(f^j(\theta))  -\phi(f^j(x))) = O(1).$$
This concludes the proof. 
\end{proof}

\begin{proof}[Proof of Proposition \ref{prop_zygmund}]
By Proposition \ref{prop_alpha},
$$
\alpha(x+h)-\alpha(x)=-\sum_{n=1}^{N(x,h)} \frac{v'(f^{n-1}(x))}{Df(f^{n-1}(x))} + \frac{\alpha(f^{n-1}(x))D^2f(f^{n-1}(x))}{Df(f^{n-1}(x))} \ h + O(h).
$$
And
$$
\alpha(x-h)-\alpha(x)=\sum_{n=1}^{N(x,h)} \frac{v'(f^{n-1}(x))}{Df(f^{n-1}(x))} + \frac{\alpha(f^{n-1}(x))D^2f(f^{n-1}(x))}{Df(f^{n-1}(x))} \ h + O(h).
$$
Therefore,
$$
|\alpha(x+h)+\alpha(x-h)-2\alpha(x)|\le K|h|,
$$
which completes the proof.
\end{proof}

\begin{lemma}\label{lema_pressao}
Let 
$$
\phi(x)=- \left(\frac{v'(x)+\alpha(x)D^2f(x)}{Df(x)}\right).
$$
Then
\begin{equation} \label{zeroint}
\int \phi \ d\mu =0,
\end{equation} 
where $\mu$ is the unique absolutely continuous invariant probability of $f$.
\end{lemma}
\begin{proof} Let $\mu_f$ be the unique invariant probability of $f$ that is absolutely continuous with respect to the Lebesgue measure. Note that the function 
$$(v,f) \mapsto \int \left(\frac{v'(x)+\alpha(x)D^2f(x)}{Df(x)}\right) \ d\mu_f$$
is continuous considering the $C^1\times C^2$ topology on its domain. This easily follows from the definition of $\alpha$ and the fact that 
$$f\mapsto \mu_f$$
is continuous considering the strong  topology on the dual space  $(C^0)^\star(\mathbb{S}^1)$ (we believe this is a folklore result in the setting of expanding maps on the circle. See  Keller and Liverani \cite{kl} for references and  stronger results).  Since $C^\infty(\mathbb{S}^1)$ is dense in $C^k(\mathbb{S}^1)$, for every finite $k$, it is enough to show (\ref{zeroint}) when $v\in C^{\infty}$ and $f$ is a $C^\infty$ expanding map on the circle.  Since $f\colon \mathbb{R} \rightarrow \mathbb{R}$ is expanding and satisfies $f(x+1)=f(x)+d$ and $v\colon \mathbb{R} \rightarrow \mathbb{R}$ is $1$-periodic we have $f_t = f+ tv$ is a family of $C^{\infty}$ expanding maps on $\mathbb{R}$  that induces expanding maps on the circle, provided that $t$ is small enough.  We will use the same notation $f_t$ for these maps on the circle. Note that $\partial_tf_t(x)|_{t=0}=v(x).$ Since expanding maps of the circle are structurally stable,   there is a family of conjugacies $h_t$ satisfying  $f_t \circ h_t = h_t \circ f$, with $h_0(x)=x$. We have  $\partial_t h_t |_{t=0}=\alpha(x)$ (see the proof of Theorem 1 in Baladi and Smania  \cite{bs2}). Indeed one can show that 
$$\partial_t h_t(x)=\alpha_t(h_t(x)),$$
where 
$$\alpha_t(x)= -\sum_{n=1}^{\infty}\frac{v(f^{n-1}_t(x))}{Df^n_t(x)}.$$
By Proposition \ref{prop_zygmund} there is $C$ such that for every $t$  small enough
\begin{equation}\label{zyg} |\alpha_t(x+h)+\alpha_t(x-h)-2\alpha_t(x)|\leq C|h|.\end{equation} 
We can conclude that there exists $C > 0$ such that
$$|\alpha_t(x)-\alpha_t(y)|\leq C |x-y|^\epsilon.$$
By (\ref{zyg}) and  Reimann \cite[proof of Proposition 4]{ode} for every $\delta > 0$ small there exists $C > 0$ such that 
\begin{equation}\label{regh} |h_t(x)-h_t(y)|\leq C |x-y|^{1-\delta},\end{equation}
provided that $t$ is small enough. By   Baladi and Smania  \cite[Eqs. (14) and (15)]{bs1} for each $x$ the map $t\mapsto h_t(x)$ is twice differentiable and $\partial^2_t h_t(x)$ satisfies
$$\partial^2_t h_t(x) = \beta_t(h_t(x)),$$
where
$$\beta_t(x) = -\sum_{n=1}^{\infty}\frac{w_t(f^{n-1}_t(x))}{Df^n_t(x)},$$
with
$$w_t(x)= \partial^2_t f_t(x)+ 2 \partial^2_{x,t}f_t(x)\alpha_t(x)+ \partial^2_{xx} f_t(x) \alpha_t^2(x).  $$
Note that $w_t \in C^\epsilon$ and  their H\"older norm is uniformly bounded provided that $t$ is small enough.  By Baladi and Smania \cite[Proposition 2.3]{bs3} (indeed here we have a far simpler situation, once $f$ is smooth everywhere) there exists $C > 0$ such that
\begin{equation}
\label{regb} |\beta_t(x)-\beta_t(y)|\leq C |x-y|^\epsilon.
\end{equation}
We claim that the curve  $t \mapsto h_t$ is a differentiable curve at $t=0$  in the Banach space $C^{\epsilon'}$of  $\epsilon'$-H\"older functions, for every  $\epsilon' \in (0,\epsilon)$,  and  its derivative at $t=0$ is the function $\alpha$. Indeed note that
\begin{eqnarray} \label{intr} h_t(x)-x - t\alpha(x)&=& \int_0^t \partial_t h_t(x) |_{t=a}  -\alpha(x)\ da \nonumber \\
&=& \int_0^t  \partial_t h_t(x) |_{t=a} - \partial_t h_t(x) |_{t=0}  \ da \nonumber \\
&=& \int_0^t  \int_0^a \partial_t^2 h_t(x) |_{t=b}  \ db \ da \nonumber \\
&=& \int_0^t  \int_0^a \beta_b(h_b(x))  \ db \ da. 
\end{eqnarray} 
Thus, if $r_t(x)= h_t(x)- x - t\alpha(x)$ we have by (\ref{regh}) , (\ref{regb}) and (\ref{intr}) 
\begin{eqnarray} \frac{|r_t(x)-r_t(y)|}{|x-y|^{\epsilon'}}&\leq& \int_0^t  \int_0^a \frac{ |\beta_b(h_b(x)) -\beta_b(h_b(y)|}{|x-y|^{\epsilon'}} \ db \ da \nonumber  \\
&\leq&   \frac{C^\epsilon}{2}  t^2.
 \end{eqnarray} 
This proves the claim. Consider a family of potentials
$$
\psi_t(x) = -\log |Df_t(h_t(x))|.
$$ 
Using that $t \mapsto h_t$ is differentiable at $t=0$ one  can prove, with  an argument similar to that used to prove that $h_t$ is differentiable at $t=0$, that the map 
$$t\mapsto \psi_t$$
is differentiable considering the Banach space of $\epsilon'$-H\"older functions on its image, with $\epsilon' < \epsilon$.  Baladi and Smania \cite{bs3} did something similar considering the space of $p$-bounded variations on the image, with $p$ large. 
Note  that $P(f,\psi_t)= 0$ for every $t$, where $P(f,\psi)$ denotes the topological pressure of $\psi$ with respect to $f$. Therefore,
$$
\partial_tP(f,\psi_t)|_{t=0} = 0.
$$
The topological pressure with respect to $f$ is a differentiable function on the Banach space of $\epsilon'$-H\"older functions and  by classical arguments of thermodynamic formalism (see Parry and Pollicott \cite{pressao}), we have
$$
\partial_tP(f,\psi_t)|_{t=0} =\int \partial_t\psi_t|_{t=0} d\mu = -\int  \frac{v'(x)+D^2f(x)\alpha(x)}{Df(x)}\ d\mu = \int \phi \ d\mu,
$$
where $\mu$ is the equilibrium state of $f$ with respect to $-\log |Df|$, that is, the unique absolutely continuous invariant probability of $f$.
\end{proof}

Now we can prove Theorem \ref{teo_se_e_so_se}.
\begin{proof}[Proof of Theorem \ref{teo_se_e_so_se}]
 If $\sigma^2=0$ then by  Proposition 4.12 in Parry and Pollicott \cite{pressao} (see also Proposition 6.1 from Broise \cite{broise}), there is a $C^{\epsilon}$-function $u: \mathbb{S}^1 \to \R$ such that
$$
\phi = u\circ f -u.
$$
Then there is  $C>0$ such that
$$
\sup_{n,x}\left|\sum_{i=0}^{n} \phi(f^i(x)) \right| \le C.
$$
And by  Proposition \ref{prop_alpha}
$$
|\alpha(x+h) - \alpha(x)| \le \left| \sum_{i=0}^{N(x,h)} \phi(f^i(x)) \ h\right| + O(h) \le C|h|.
$$
Therefore $\alpha$ is a lipschitzian function and we can differentiate it at almost every point. Then, differentiating (\ref{tce}), we obtain:
$$
v'(x) = \alpha'(f(x))Df(x) - D^2f(x)\alpha(x) - Df(x)\alpha'(x).
$$
Therefore,
$$
\alpha'(f(x))-\alpha'(x) = \frac{v'(x) + D^2f(x)\alpha(x)}{Df(x)}=\phi(x).
$$
Since $\phi(x) = u \circ f -u$, it follows that
$$
(\alpha' - u) \circ f(x) = (\alpha' - u) (x).
$$
Since $f$ is ergodic, we can conclude that there is a constant $k$ such that $\alpha' = u + k$. Therefore $\alpha$ is of class $C^{1+\epsilon}$. Reciprocally, if $\alpha$ is of class $C^{1+\epsilon}$, we can differentiate (\ref{tce}). Thus,
$$
v'(x) = \alpha'(f(x))Df(x) - D^2f(x)\alpha(x) - Df(x)\alpha'(x).
$$
Therefore,
$$
\alpha'(f(x))-\alpha'(x) = \frac{v'(x) + D^2f(x)\alpha(x)}{Df(x)}.
$$
Let $p$ be a periodic point of $f$, that is, $f^n(p)=p$. Then
$$
\sum_{j=0}^{n-1}\phi(f^j(p)) = \alpha'(f^n(p)) - \alpha'(p) =\alpha'(p) - \alpha'(p) =0.
$$
Therefore, by Livsic Theorem ( Livsic \cite{livsic}. See also  Parry and Pollicott \cite{pressao}), there is a $\epsilon$-H\"older function $u$ such that
$$
\phi=u\circ f - u.
$$
Hence, as we can see in  Proposition 4.12 in Parry and Pollicott \cite{pressao} (see also Proposition 6.1 from Broise \cite{broise}), $\sigma^2 = 0$.
\end{proof}
\begin{proof}[Proof of Corollary  \ref{cor_alpha}] By  Livsic \cite{livsic}, there exists $u \in C^\epsilon$ such that $\phi= u\circ f - u$ if and only if for every periodic point $x$ we have
$$\sum_{i=0}^{p-1} \phi(f^i(x))=0,$$
where $p$ is the period of $x$. But by   Proposition 4.12 in Parry and Pollicott \cite{pressao} (see also Proposition 6.1 from Broise \cite{broise}) such $u$ exists if and only if $\sigma^2=0$. Now we can apply Theorem \ref{teo_se_e_so_se} to conclude the proof. 
\end{proof}

\section{Proof of Central Limit Theorem for the modulus of continuity of $\alpha$}\label{sec_prova}

We are going to need the following 

\begin{lemma}\label{lema_exp_lyapunov}
For $\mu$-a.e. $x \in \mathbb{S}^1$,
$$\lim_{h\to 0} \frac{N(x,h)}{-\log |h|} = \frac{1}{L},$$
where $L := \int \log |Df| d\mu$
is the Lyapunov exponent of $f$. \index{Lyapunov exponent}
\end{lemma}
\begin{proof}
By Ergodic Birkhoff's Theorem, for $\mu$-a.e. $x \in \mathbb{S}^1$
\begin{align*}
\lim_{n\to \infty} \frac{1}{n}\log|Df^{n}(x)|&=\lim_{n\to \infty} \frac{1}{n}\log \prod_{j=0}^{n-1}|Df(f^j(x))|\\
&= \lim_{n\to \infty} \frac{1}{n}\sum_{j=0}^{n-1}\log |Df(f^j(x))| = \int \log |Df(x)| \ d\mu= L.
\end{align*}
By  (\ref{defNxh}) we have
$$
|Df^{N(x,h)}(x)| \le \frac{1}{|h|} \le |Df^{N(x,h)+1}(x)|.
$$
Therefore,
$$
\frac{1}{N(x,h)}\log |Df^{N(x,h)}(x)| \le \frac{1}{N(x,h)}\log \frac{1}{|h|} \le \frac{1}{N(x,h)}\log |Df^{N(x,h)+1}(x)|.
$$
Hence, for $\mu$-a.e. $x \in \mathbb{S}^1$,
$$
\frac{1}{N(x,h)}\log \frac{1}{|h|} \stackrel{N(x,h) \to \infty}\longrightarrow \int \log |Df(x)| d\mu = L.
$$
\end{proof}

\begin{proof}[Proof of Theorem \ref{teo_modulo_alpha}]
It is enough to prove that for every sequence
$
h_n \underset{n\to \infty}{\longrightarrow}0,
$
we have
$$
\lim_{n\to \infty}\mu \left\{x:\frac{\alpha(x+h_n)-\alpha(x)}{\sigma \ell h_n\sqrt{-\log |h_n|}} \le y\right\} = \frac{1}{\sqrt{2\pi}} \int_{-\infty}^{y} e^{-\frac{t^2}{2}} dt.
$$
By Proposition \ref{prop_alpha},
\begin{align*}
\alpha(x+h_n)-\alpha(x) = h_n\sum_{i=1}^{N(x,h_n)} \phi(f^{i-1}(x)) + O(h_n).
\end{align*}
Since 
$$
\frac{1}{\sigma \ell h_n\sqrt{-\log |h_n|}}\ O(h_n)\stackrel{n\to \infty} \longrightarrow 0,
$$
we have
\begin{align*}
\frac{\alpha(x+h_n)-\alpha(x)}{\sigma \ell h_n\sqrt{-\log|h_n|}} = \frac{1}{\sigma \ell \sqrt{-\log|h_n|}}\sum_{i=1}^{N(x,h_n)} \phi\left( f^{i-1}(x)\right) + r(x,h_n),
\end{align*}
where
$$\lim_n \sup_{x\in [0,1]}|r(x,h_n)|=0.$$
From now on the proof is similar to an argument in Leplaideur  and Saussol \cite{rb}. We will include it here for the sake of completeness.  Let us define
$$
X_N(\theta,x) = \frac{1}{\sigma\sqrt N}\sum_{k=0}^{\lfloor N\theta\rfloor -1} \phi(f^k(x)) + \frac{(N\theta-\lfloor N\theta  \rfloor)}{\sigma \sqrt{N} } \ \phi(f^{\lfloor N \theta  \rfloor}(x)),
$$
and $Y_n$ by:
$$
Y_n(\theta,x)= \frac{1}{\sigma \sqrt{\nu_n(x)} }\sum_{k=0}^{\lfloor \nu_n \theta  \rfloor-1}  \phi(f^k(x))+ \frac{(\nu_n(x)\theta-\lfloor \nu_n(x) \theta \rfloor)}{\sigma \sqrt{\nu_n(x)} } \ \phi(f^{\lfloor \nu_n(x) \theta \rfloor}(x))
$$ 
where $\nu_{n}(x) = N(x,h_n)$.
By Lemma \ref{lema_exp_lyapunov}, 
$$
\frac{N(x,h)}{-\log|h|} \stackrel{h\to 0}\longrightarrow \frac{1}{L},
$$
then
$$
\frac{N(x,h_n)}{-\log|h_n|} \stackrel{P}\longrightarrow \frac{1}{L}.
$$
By Lemma \ref{lema_pressao} and  Keller \cite{kellerbv} and Hofbauer and Keller \cite{hofbauer} we have that $X_N(\theta,x)$ converges in distribution to the Wiener Process. We denote this convergence by
$$X_N(\theta,x) \stackrel{D}\longrightarrow_N W.$$
Then, \cite[page 152]{bil} we conclude that
$$Y_n(\theta,x) \stackrel{D}\longrightarrow_n W,$$
where $W$ is the Wiener process. Hence, taking $\theta=1$ we conclude that 
$$Y_n(1,x) \stackrel{D} \longrightarrow_n \mathcal{N}(0,1),
$$ where $\mathcal{N}(0,1)$ denotes the Normal distribution with average zero and  variance one. 
Observe that 
$$
Y(1,x) = \frac{1}{\sigma \sqrt{\nu_n(x)} }\sum_{k=0}^{N(x,h)-1}  \phi(f^k(x)).
$$
Therefore, considering 
$$
Z_n(x) = \frac{\sqrt{N(x,h_n)}}{\sqrt{-\log |h_n|}},
$$
by Slutsky's theorem (see \cite{Gut}), since 
$$
Z_n \stackrel{P} \longrightarrow_n \ell,
$$
we can conclude that
$$
\tilde{Y}_n(x)=Y_n(1,x).Z_n(x)=\frac{1}{\sigma \sqrt{-\log |h_n|} }\sum_{k=0}^{N(x,h)-1}  \phi(f^k(x))\stackrel{D} \longrightarrow_n \ell \ \mathcal{N}(0,1).
$$
Hence, taking $R_n(x) = r(x,h_n)$ and using Slutsky's theorem one more time, we have
$$
\frac{1}{\ell}\tilde{Y}_n(x) + R_n(x) = \frac{\alpha(x+h_n)-\alpha(x)}{\sigma \ell h_n\sqrt{-\log|h_n|}}\stackrel{D} \longrightarrow_n  \mathcal{N}(0,1).
$$
\end{proof}

\begin{proof}[Proof of Corollary \ref{cor_lipschitz}]
The proof is identical to the proof of a  similar statement in \cite{lsx}.
\end{proof}

\begin{proof}[Proof of Theorem \ref{teo_lil}] By   Keller \cite{kellerbv} and Hofbauer and Keller \cite{hofbauer} (see also Przytycki,  Urba{\'n}ski and  Zdunik \cite{puz1} for the analytic setting ) we have that 
$$\limsup_{N\rightarrow \infty} \frac{\sum_{i=0}^{N-1} \phi(f^{i}(x))}{\sqrt{2N \log \log N}} = \sigma(\phi).$$
and
$$\liminf_{N\rightarrow \infty} \frac{\sum_{i=0}^{N-1} \phi(f^{i}(x))}{\sqrt{2N \log \log N}} =- \sigma(\phi).$$
for $\mu$-almost every point $x$. By Proposition \ref{prop_alpha} and Lemma \ref{lema_exp_lyapunov} the result easily follows.
\end{proof}

\section{Dichotomy for the regularity of $\alpha$}\label{sec2.2}

We will prove Theorem \ref{teo_alpha_dif} using methods similar to those in Heurteaux \cite{yanick}.  We need to introduce some notations and definitions. Given a function $w\colon \mathbb{R}\rightarrow \mathbb{R}$ we define the second-order difference of $w$ by
$$
\Delta_h^2 w (x) = w(x+h) + w(x-h) -2w(x).
$$
Denote
$$\omega(w,x,h)= \frac{\Delta_h^2 w (x)}{|h|^{1+\epsilon}}.$$

\begin{lemma}\label{lema_tecnico_delta} We have
\begin{equation}\label{ocr} \frac{|\Delta_h^2\alpha(x)|}{|h|^{1+\epsilon}} = |Df(x)|^\epsilon  \frac{|\Delta_{Df(x)h}^2\alpha(f(x))|}{|Df(x) h|^{1+\epsilon}} + O(1).\end{equation}
\end{lemma}
\begin{proof}  In Corollary \ref{prop_zygmund} we saw that $\alpha$ is  in the  Zygmund class $\Lambda_1$. So $\alpha$ is   $\beta$-H\"older for every  $0<\beta<1$ (see \cite{donaire} and references therein).  Consequently
\begin{align*} \alpha(f(x+h))&=\alpha(f(x)+Df(x)h+ O(|h|^{2}) )\\
&= \alpha(f(x)+Df(x)h)+ O(|h|^{1+\epsilon}).
\end{align*}
on the other hand
\begin{align*} \alpha(f(x+h))&=v(x+h)+Df(x+h)\alpha(x+h)\\
&=v(x)+Dv(x)h +Df(x)\alpha(x+h) + D^2f(x)\alpha(x+h)h+ O(|h|^{1+\epsilon}) \\
&=v(x) +[Dv(x) +D^2f(x)\alpha(x)]h  +Df(x)\alpha(x+h) + O(|h|^{1+\epsilon}).
\end{align*}
So
\begin{align*}
\Delta_{Df(x)h}^2\alpha(f(x))= Df(x) \Delta_h^2\alpha(x)  + O(|h|^{1+\epsilon}).
\end{align*}
and (\ref{ocr}) follows.
\end{proof}
Therefore there is $K > 0$ such that 
\begin{equation}\label{modulo_delta}\omega(\alpha,x,h) \geq |Df(x)|^\epsilon \omega(\alpha,f(x),Df(x) h) -K,\end{equation}
for every $x$ and $h \neq 0$. We will denote by
$b= \inf\{|Df(x)|: x\in \R\}$ and $B= \sup\{|Df(x)|: x\in \R\}$.
It is easy to see that
\begin{lemma}\label{lema2} Let $K$ be as in (\ref{modulo_delta}).
Let $x \in \R$, $h > 0$ and $L > 0$  satisfying
$$
 \omega(\alpha,f(x),Df(x) h) \ge \frac{K+L}{b^{\epsilon}-1}.
$$
Then
$$
\omega(\alpha,x,h)  \ge \frac{K+L|Df(x)|^{\epsilon}}{b^{\epsilon}-1},
$$
\end{lemma}

\begin{proof}[of Theorem \ref{teo_alpha_dif}] Note that 
$$
\sup \left\{\omega(\alpha,x,h), x\in \R, h>  0 \right\} < \infty.
$$
if and only if  $\alpha$ is of class $C^{1+\epsilon}$  (see \cite[Lemma $5.4$, page $207$]{krantz}  for more details). Suppose that $\alpha$ is {\it not} $C^{1+\epsilon}$. Then there exists $x_0\in [0,1], h_0\in (0,1)$ and $L > 0$ such that 
$$\omega(\alpha,x_0,h_0) > \frac{K+L}{b^{\epsilon}-1}.$$
Given $n \in \mathbb{N}$, let $x_i$ be such that $f(x_{i+1})=x_i$ for every $i < n$. Then by  Lemma \ref{lema2} we have
$$\omega(\alpha,x_n,\frac{h_0}{Df^n(x_n)})\geq \frac{L}{b^{\epsilon}-1}|Df^n(x_n)|^\epsilon,$$
that is
\begin{equation}\label{weak}  |\Delta_{\frac{h_0}{Df^n(x_n)}}^2\alpha(x_n)| \geq \frac{L|h_0|^{\epsilon}}{b^{\epsilon}-1}\frac{h_0}{Df^n(x_n)}.\end{equation}
Fix  $y \in [0,1]$. For each  $h > 0$ let $n=n(h)$ be minimal such that 
$$f^n [y -h, y+h] \supset [x_0-h_0,x_0+h_0].$$
Then by Lemma \ref{lema_distorcao} there is $C$ such that for every $h >0$ we have
$$\frac{1}{C} \leq  \frac{Df^n(x)}{Df^n(y)} \leq C$$
and 
$$\frac{h}{C}   \leq  \frac{h_0}{Df^n(x)} \leq C h$$
for every $x \in [y -h, y+h]$. Choose $x_n \in [y -h, y+h] $ such that $f^n(x_n)=x_0$. If $\alpha$ is differentiable at $y$ we have that 
$$ |\Delta_{\frac{h_0}{Df^n(x_n)}}^2\alpha(x_n)|= o(h)=o(\frac{h_0}{Df^n(x_n)}), $$
This contradicts   (\ref{weak}).
\end{proof} 

\bibliographystyle{abbrv}


\end{document}